\documentclass[11pt]{article}
\usepackage{amsmath,amsthm,amsfonts,amssymb,latexsym}
\usepackage[utf8]{inputenc}
\usepackage[english]{babel}

\newtheorem{theorem}{Theorem}
\newtheorem{proposition}[theorem]{Proposition}

\newtheorem{corollary}{Corollary}[theorem]

\def\N{{\mathbb N}}
\newcommand{\ps}{\smallbreak}
\newcommand{\tq}{:}
\newcommand{\la}{\langle}
\newcommand{\ra}{\rangle}
\newcommand{\eps}{\varepsilon}

\newcommand{\xb}{\bar{x}}
\newcommand{\yb}{\bar{y}}
\newcommand{\xs}{x^*}
\newcommand{\ys}{y^*}

\newcommand{\tom}{\rightrightarrows}
\newcommand{\conv}{{\rm conv} \kern.15em}
\newcommand{\cconv}{{\overline{\rm conv}} \kern.15em}

\begin{document}
\thispagestyle{empty}
\begin{center}
{\large\bf
Convex KKM maps, monotone operators and Minty variational inequalities}
\end{center}

\begin{center}
  {\small\begin{tabular}{c}
  Marc Lassonde\\
   Universit\'e des Antilles, 97159 Pointe \`a Pitre, France\\
  E-mail: marc.lassonde@univ-ag.fr
  \end{tabular}}
\end{center}

\medbreak\noindent
\textbf{Abstract.}
It is known that for convex sets, the KKM condition is equivalent
to the finite intersection property.
We use this equivalence to obtain
a characterisation of monotone operators in terms of convex KKM maps and
in terms of the existence of solutions to Minty variational inequalities.
The latter result provides a converse to the seminal theorem of Minty.
\medbreak\noindent
  
\medbreak\noindent

\section{Introduction}\label{introduction}
More than twenty years ago, in two joint papers with A. Granas
\cite{GL91,GL95} 
the authors presented a new geometric approach in convex analysis.
This approach was based on the finite intersection property
of KKM-maps with closed \textit{convex} values.
It was shown that this special case of the KKM Principle admits an elementary direct proof
and yet, as the general KKM Principle,
it has numerous applications in different areas of mathematics.
\ps
In this note, we further enlightened the above method by showing that
it is intimately connected with the Minty-Browder monotonicity method.
More precisely, we show that an operator $T:E\tom E^*$ is monotone if and only if
for every $x^*\in E^*$, the natural convex-valued map
$\Gamma_{T-x^*}:E\tom E$ associated to it, namely
$$
\Gamma_{T-x^*}(y):=\{ x\in E\tq \la y^*-x^*, y-x\ra\ge 0,~ \forall y^*\in T(y)\},
$$
is KKM. Then we show that the Minty Variational Inequality
associated to an operator $T:E\tom E^*$ has a solution
for every nonempty compact convex $K\subset E$ and every $x^*\in E^*$,
that is 
$$\exists \xb\in K:\ \forall (y,y^*)\in T\cap(K\times E^*),\ \la y^*-x^*,y-\xb\ra\ge 0,
$$
if and only if $T$ is monotone.
This last result provides a converse to Minty's theorem \cite{Min62,Min67}.
\ps
We should mention that our interest in writing up this material was stimulated by a
paper by John \cite{Joh01} where it is shown that the Minty Variational Inequality
associated to $T:E\tom E^*$ has a solution for every nonempty compact convex $K\subset E$
and $x^*=0$
if and only if $T$ is properly quasimonotone.
\section{Convex KKM condition vs. Finite Intersection Property}\label{KKM-fip}
Set-valued maps $T : X \tom Y$ between sets $X$ and $Y$ are identified with their
graphs $T \subset X \times Y$, so $y\in T(x)$ is equally written as
$(x, y) \in T$. The \textit{values} of
$T : X \tom Y$ are the subsets $T(x)\subset Y$ for $x\in X$ and
the \textit{domain} of $T$ is the set $D(T):=\{ x\in X\tq T(x)\ne\emptyset\}$.
\ps 
In the sequel, $E$ denotes a real locally convex topological vector space,
$E^*$ its dual space and $\la ., .\ra$ the duality mapping.
For $A \subset E$, we use the abbreviation
$[A] = \cconv A$ for the closed convex hull of $A$.
For each positive integer $n$, we set $[n] =\{\,i\in \N \tq 1\le i \le n \,\}$.
A subset of $E$ is said to be \textit{finitely closed} if
its intersection with any finite-dimensional subspace of $E$ is closed (for
the Euclidean topology).
\ps
A set-valued map $\Gamma:E\tom E$ is called a {\it KKM-map} provided it satisfies
\begin{center}
(KKM) ~~ For every finite subset
$A\subset D(\Gamma),\ [A]\subset\bigcup\{\,\Gamma(x)\tq x\in A\,\}$.
\hspace*{\fill}{}
\end{center}
\ps
The KKM Principle asserts that KKM-maps with closed values has the finite intersection property. This intersection principle
is known to be equivalent to Sperner's combinatorial lemma
and to Brouwer's fixed point theorem (see e.g.\ Granas-Dugundji's monograph
\cite{GD03}).

When the KKM-maps have \textit{convex} values, the KKM Principle
can be given an elementary proof
(see e.g. Valentine \cite[p. 76]{Val76} or Granas-Lassonde \cite{GL91,GL95}).
Moreover in this convex case, as was observed by John \cite{Joh01},
the finite intersection property is actually equivalent to the KKM condition.
For the sake of completeness,
we provide a proof of this fundamental equivalence result.

\begin{theorem}\label{Th-KKM-fip}
Let $E$ be a vector space and let $\Gamma:E\tom E$ be a set-valued map
with finitely closed and convex values.
The following are equivalent:
\ps
{\rm (KKM)} ~~For every finite subset $A\subset D(\Gamma),\ [A]
\subset\bigcup\{\,\Gamma(x)\tq x\in A\,\}$;
\ps
{\rm ~~(FIP)} ~~For every finite subset $A\subset D(\Gamma),\ [A]
\cap\bigcap\{\,\Gamma(x)\tq x\in A\,\}\ne \emptyset$.
\end{theorem}

\begin{proof}
(KKM) $\Rightarrow$ (FIP) (see \cite{GL91,GL95}).
The proof is by induction on the cardinality of the finite sets $A$.
For any set consisting of a single element,
both statements (KKM) and (FIP) are the same.
Assuming that (FIP) holds for any set containing $(n-1)$ elements, we consider
a subset $A = \{x_1,x_2,\ldots,x_n\}\subset D(\Gamma)$ with $n$ elements.
Let $G_i=\Gamma(x_i)\cap [A]$.
We have to show that the family $\{G_i\tq i\in [n]\}$ has a nonempty intersection.

Observe that the sets $G_i$ are contained in the finite dimensional vector space
spanned by $A$. We may therefore assume
that the underlying space is finite dimensional, the sets $G_i$ are closed
and the topology is described by a norm $\|.\|$.
For a point $y$ and a set $G$, we let $d(y,G):=\inf \{\|y-z\|\tq z\in G\}$. 

For each $j\in [n]$, by induction hypothesis we may pick up a point
$y_j\in \bigcap\{G_i \tq i\ne j\}$.
Let $K = [y_1, y_2, \ldots, y_n]$.
The continuous function $f:y\mapsto \max\{d(y,G_i) \tq i\in [n]\}$
attains its minimum on the compact set $K$ at a point $\yb$.
Since the sets $G_i$ are closed, to prove the result it suffices to show that
$f(\yb)=0$. Suppose to the contrary that $f(\yb)=\eps>0$. 
\ps
It follows from (KKM) that $\bigcup\{G_i\tq i\in [n]\}=[A]$ is a convex set containing the points $y_1,y_2,\ldots,y_n$, so it also contains the point $\yb\in K$.
Without loss of generality, we may assume that  
$\yb$ belongs to $G_n$, so that $d(\yb,G_n)=0$. 
The function $y\mapsto d(y,G_n)$ being continuous, there is a point
close to $\yb$ of the form $y_t=t\yb+(1-t)y_n\in K$ with $0\le t< 1$
such that $d(y_t,G_n)<\eps$.
On the other hand, $y_n\in G_i$ for all $i\in [n-1]$,
hence $d(y_n,G_i)=0$ for all $i\in [n-1]$.
From the convexity of the functions $y\mapsto d(y,G_i)$
we derive that for all $i\in [n-1]$, we have
$d(y_t,G_i)\le t d(\yb,G_i)\le tf(\yb)<f(\yb)$.
Thus, the point $y_t\in K$ would verify $d(y_t,G_i)<f(\yb)$ for all $i\in [n]$,
that is, $f(y_t)<f(\yb)=\min\{f(y)\tq y\in K\}$, which is a absurd.
\medbreak
(FIP) $\Rightarrow$ (KKM) (see \cite{Joh01}).
The proof is also by induction on the cardinality of the finite sets $A$.
As already noticed, for any set consisting of a single element,
both statements (KKM) and (FIP) are the same.
Assume that (KKM) holds for all sets $A$ with $n-1$ elements
and consider a set $A = \{x_1,x_2,\ldots,x_n\}\subset D(\Gamma)$ with $n$ elements.
By (FIP), choose $\xb$ in $[A]\cap \bigcap\{\,\Gamma(x_i)\tq i\in [n]\,\}$.
Let $x\in [A]$ with $x\ne \xb$.
Consider $z$ on the boundary of $[A]$
such that $x\in [z,\xb]$. Since $z\in [A\setminus\{x_i\}]$
for some $i$, from the induction hypothesis we derive that
$z\in \Gamma(x_{i_0})$ for some $i_0\in [n]\setminus\{i\}$.
Now, since $\xb\in \Gamma(x_{i_0})$ and $\Gamma(x_{i_0})$ is convex,
we infer that $x\in [z,\xb]\subset \Gamma(x_{i_0})$.
Therefore, every $x\in [A]$ belongs to $\bigcup\{\,\Gamma(x_i)\tq i\in [n]\,\}$.
\end{proof}

\section{KKM maps vs. monotone operators}\label{KKM-monotone}

A subset $T\subset E\times E^*$, or set-valued $T:E\tom E^*$, is said to be
{\it monotone\/} provided
\begin{equation*}\label{monotone}
\forall (x,\xs)\in T,\ \forall (y,\ys)\in T,\ \la \ys-\xs,y-x\ra\ge 0,
\end{equation*}
and {\it  quasimonotone\/} provided
\begin{equation*}\label{quasimonotone}
\forall (x,\xs)\in T,\ \forall (y,\ys)\in T,\ \max\{\la \xs,x-y\ra, \la \ys,y-x\ra\}\ge 0.
\end{equation*}
Given $T:E\tom E^*$, we define $\Gamma_T:E\tom E$ by
$$
\Gamma_T(y)=\{ x\in E\tq \la y^*, y-x\ra\ge 0,~ \forall y^*\in T(y)\}.
$$
Observe that the sets $\Gamma_T(y)$ are convex and finitely closed and
$\Gamma_T(y)=E$ when $y\not\in D(T)$.
\ps
The relationship between these notions is described in the following proposition:

\begin{proposition}
Let $E$ be a real locally convex topological vector space
with topological dual $E^*$.
Let $T:E\tom E^*$. Then:
\ps
{\rm (a)} $T$ monotone $\Rightarrow$ $\Gamma_T$ KKM
$\Rightarrow$ $T$ quasimonotone.
\ps
{\rm (b)} $T$ is monotone $\Leftrightarrow$ $\forall x^*\in E^*$,
the operator $x\mapsto T(x)-x^*$ is quasimonotone.
\end{proposition}

\begin{proof}
These facts are well-known; we give the proof for the sake of completeness.
\ps
(a1) (see \cite{GL91})
We show: \textit{$T$ monotone $\Rightarrow$ $\Gamma_T$ KKM}.
Let $\{y_1, \ldots, y_n\} \subset D(T)$. Consider
$x_0 = \sum_{i=1}^n\lambda_i y_i$, where
$\lambda_i\ge 0$ for $i\in [n]$ and $\sum_{i=1}^n\lambda_i=1$.
For $(x,y) \in E\times D(T)$, set
$$g(x,y) := \sup_{y^*\in T(y)} \la y^*,x-y\ra.$$
By monotonicity of $T$, we have
$$g(y_i,y_j) + g(y_j,y_i) \leq 0,\quad\forall i, j \in [n],$$
hence
$$
\sum_{i=1}^n\lambda_i g(y_i,y_j) + \sum_{i=1}^n\lambda_i g(y_j,y_i)
\leq 0,\quad\forall j \in [n],$$
and by convexity of $x \mapsto g(x,y)$,
$$g(x_0,y_j)+\sum_{i=1}^n\lambda_i g(y_j,y_i) \leq 0,\quad\forall j \in [n].$$
Applying the same operations on these inequalities 
(multiplying by $\lambda_j$, summing over $j$, using the convexity of
$x \mapsto g(x,y)$), we arrive at
$$\sum_{j=1}^n\lambda_j g(x_0,y_j) + \sum_{i=1}^n\lambda_i g(x_0,y_i) \leq 0.$$
Thus, $g(x_0,y_i) \leq 0$ for at least one $i\in [n]$.
This means that
$x_0 \in \bigcup\{\Gamma_T(y_i)\tq i \in [n]\}$ and proves that $\Gamma_T$ is KKM.
\ps
(a2) We show: \textit{$\Gamma_T$ KKM $\Rightarrow$ $T$ quasimonotone}.
Let $(x,\xs)$ and $(y,\ys)$ in $T$. Consider $z=(x+y)/2\in [x,y]$.
Since $\Gamma_T$ is KKM, we must have either $z\in \Gamma_T(x)$
or $z\in \Gamma_T(y)$. The first case implies $\la \xs,x-z\ra\ge 0$,
hence $\la \xs,x-y\ra\ge 0$,
the second one implies $\la \ys,y-z\ra\ge 0$,
hence $\la \ys,y-x\ra\ge 0$; therefore
always  $\max\{\la \xs,x-y\ra, \la \ys,y-x\ra\}\ge 0$.
\ps
(b) (see \cite{ACL94})
If $T$ is monotone, then for every $x^*\in E^*$,
the operator $x\mapsto T(x)-x^*$ is clearly monotone,
hence quasimonotone. To prove the converse,
let $x,y$ in $D(T)$ with $x\neq y$, let $x^*\in T(x)$,
$y^*\in T(y)$, and let $\eps>0$. Choose $z^*\in E^*$ such that
 $$\langle x^*-z^*,y-x\rangle=\eps>0.$$
Since $x\mapsto T(x)-z^*$ is assumed to be quasimonotone,
the above inequality implies that
$\langle y^*-z^*,y-x\rangle\ge 0$,
or equivalently
$\langle y^*,y-x\rangle
 \ge \langle z^*,y-x\rangle= \langle x^*,y-x\rangle -\eps$,
that is, $\langle y^*-x^*,y-x\rangle \ge -\eps$.
Since $\eps$ can be arbitrarily small, we conclude that $T$ is monotone.
\end{proof}

As a consequence of the previous proposition, we readily obtain
a characterization of monotone operators in terms of KKM maps:

\begin{theorem}\label{Th-KKM-monotone}
Let $E$ be a real locally convex topological vector space
with topological dual $E^*$.
Let $T:E\tom E^*$. The following are equivalent:
\ps
{\rm (1)} $T$ is monotone, which amounts to:
for every finite subset $\{(x_i,x^*_i)\tq i\in [m]\}\subset T$,
\begin{equation*}
\forall i,j\in [m],~\la x^*_i - x^*_j, x_i - x_j\ra \ge 0;
\end{equation*}
\ps
{\rm (2)} For every $x^*\in E^*$, the map $\Gamma_{T-x^*}$ is KKM, that is:
for every $x^*\in E^*$ and for every finite subset $\{x_i\tq i\in [m]\}\subset D(T)$,
\begin{equation*}\label{Gamma-T-KKM}
\forall \xb\in [x_1,\ldots,x_m],\ \exists i\in [m] :
\forall  x^*_i \in T(x_i),~\la x^*_i - x^*, x_i - \xb\ra \ge 0.
\end{equation*}
\end{theorem}

\section{Monotone operators vs.\ Minty Variational Inequalitiy}\label{monotone-Minty}

Let $T\subset E\times E^*$ and $x^*\in E^*$.
The {\it Minty Variational Inequality} governed by $T$ and $x^*$ 
is the problem of finding a solution $\xb\in [D(T)]$
to the following system of linear equalities:
\begin{center}
MVI\,($T$, $x^*$)~~~
$\forall (y,y^*)\in T,\ \la y^*-x^*,y-\xb\ra\ge 0.$\hspace*{\fill}{ }
\end{center}
Minty's seminal theorem \cite{Min62,Min67} asserts that every finite
or compact subsystem of MVI\,($T$, $x^*$) has a solution whenever $T$ is monotone.
The next result provides a converse to Minty's theorem.

\begin{theorem}\label{converse-Minty}
Let $E$ be a real locally convex topological vector space
with topological dual $E^*$.
Let $T:E\tom E^*$. The following are equivalent:
\ps
{\rm (1)} $T$ is monotone, that is:
for every finite subset $\{(x_i,x^*_i)\tq i\in [m]\}\subset T$,
\begin{equation*}
\forall i,j\in [m],~\la x^*_i - x^*_j, x_i - x_j\ra \ge 0;
\end{equation*}
\ps
{\rm (2)} For every $x^*\in E^*$
and every finite subset $\{(x_i,x^*_i)\tq i\in [m]\}\subset T$,
\begin{equation*}
\exists \xb\in [x_1,\ldots,x_m]: \forall i\in [m],
\ \la x^*_i - x^*, x_i - \xb\ra \ge 0.
\end{equation*}
\ps
{\rm (3)} For every $x^*\in E^*$
and every nonempty compact convex subset $K\subset E$,
\begin{equation*}
\exists \xb \in K:
\forall (y,y^*)\in T\cap (K\times E^*),\ \la y^*-x^*,y-\xb\ra\ge 0.
\end{equation*}
\end{theorem}

\begin{proof}
(1) $\Leftrightarrow$ (2). The statement
(1) can be rephrased as "any finite subset $T'\subset T$
is monotone", which
is equivalent by Theorem \ref{Th-KKM-monotone} to the statement:
"for any $x^*\in E^*$ and any finite subset $T'\subset T$,
the map $\Gamma_{T'-x^*}$ is KKM".
Since the map $\Gamma_{T'-x^*}$ has convex and finitely closed values in $E$,
the latter statement is equivalent by Theorem \ref{Th-KKM-fip} to the following:
"for any $x^*\in E^*$ and any finite subset $T'\subset T$,
the map $\Gamma_{T'-x^*}$ satisfies (FIP)", which
is clearly a restatement of (2).
\ps
(2) $\Rightarrow$ (3). Let $x^*\in E^*$ and let
$K$ be a nonempty compact convex subset of $E$.
For $(y,y^*)\in T\cap (K\times E^*)$, consider the sets
$$\Gamma(y,y^*):=\{ x\in K\tq \la y^*-x^*, y-x\ra\ge 0\}.$$
It follows from (2) that for every finite subset
$\{(x_i,x^*_i)\tq i\in [m]\}\subset T\cap (K\times E^*)$
there exists $\xb\in [x_1,\ldots,x_m]$ such that
$ \forall i\in [m],\ \la x^*_i - x^*, x_i - \xb\ra \ge 0$.
Since $K$ is convex and the $x_i$ are in $K$, we derive that $\xb$
lies in $K$. Hence, (2) implies that "the family $\{\Gamma(y,y^*) : (y,y^*)
\in T\cap (K\times E^*)\}$ has the finite intersection property".
But since the sets $\Gamma(y,y^*)$ are closed in the compact set $K$,
this is equivalent to saying that 
"the family $\{\Gamma(y,y^*) : (y,y^*)\in T\cap (K\times E^*)\}$ has a
nonempty intersection", which is a restatement of (3).
\ps
(3) $\Rightarrow$ (2). Let $x^*\in E^*$ and let
$\{(x_i,x^*_i)\tq i\in [m]\}\subset T$. Apply (3) with $K= [x_1,\ldots,x_m]$
to obtain (2).
\end{proof}

As in John \cite{Joh01} for the case of quasimonotonicity,
from Theorem \ref{converse-Minty} we derive a very simple
characterization of monotonicity:
\begin{corollary}
Let $E$ be a real locally convex topological vector space
with topological dual $E^*$.
Let $T:E\tom E^*$. The following are equivalent:
\ps
{\rm (1)} $T$ is monotone, that is:
for every $(x_1,x^*_1)\in T$ and $(x_2,x^*_2)\in T$,
\begin{equation*}
\la x^*_2 - x^*_1, x_2 - x_1\ra \ge 0;
\end{equation*}
\ps
{\rm (2)} For every $x^*\in X^*$ and every $(x_1,x^*_1)\in T$ and $(x_2,x^*_2)\in T$,
\begin{equation*}
\exists \xb\in [x_1,x_2]:
\la x^*_1 - x^*, x_1 - \xb\ra \ge 0 \text{ and } \la x^*_2 - x^*, x_2 - \xb\ra \ge 0.
\end{equation*}
\end{corollary}

\begin{proof}
Obviously, an operator $T:E\tom E^*$ is monotone if and only if
its restriction to any closed interval in $E$ is monotone.
Since nonempty compact convex subsets of closed intervals in $E$ are closed intervals
$[x_1,x_2]$, the result follows from Theorem \ref{converse-Minty}.
\end{proof}
\medbreak\noindent
\textit{Acknowledgement.}
This work was completed while the author was visiting the
Vietnam Institute for Advanced Study in Mathematics (VIASM).
He would like to thank the VIASM for financial support and hospitality.

\end{document}